\documentclass[12pt]{article}
\usepackage{enumerate,amsfonts,dsfont,amsmath,accents,amsthm,mathrsfs,cases,bbm,amssymb,color}
\usepackage{indentfirst} %首段首行缩进
\usepackage[bookmarks]{hyperref}      %书签
\usepackage{fancyhdr}

\allowdisplaybreaks      %多行公式自动分页
\newtheorem{theorem}{Theorem}[section]
\newtheorem{lemma}[theorem]{Lemma}
\newtheorem{proposition}[theorem]{Proposition}
\newtheorem{corollary}[theorem]{Corollary}
\theoremstyle{definition}
\newtheorem{definition}[theorem]{Definition}
\newtheorem{examples}[theorem]{Examples}

\newtheorem{remark}[theorem]{Remark}
\numberwithin{equation}{section}

\newcommand{\s}{\mathrm{sign\,}}

\newcommand{\adj}{\mathrm{adj\,}}
\def\geq{\geqslant}
\def\leq{\leqslant}
\numberwithin{equation}{section}
\begin{document}
\title{Currents carried by the  subgradient graphs  of semi-convex functions and applications to Hessian measure}
\author{Qiang Tu\and Wenyi Chen\and
School of Mathematics and Statistics, Wuhan University
Wuhan 430072, China\footnote{\emph{Email addresses}:~qiangtu@whu.edu.cn(Qiang Tu), wychencn@whu.edu.cn(Wenyi Chen).}.
}
\date{}   %不显示日期
\maketitle
\noindent{\bf Abstract:}  In this paper we study   integer multiplicity rectifiable currents
carried by the subgradient (subdifferential) graphs   of  semi-convex functions on a $n$-dimensional convex domain,
and show a weak continuity theorem with respect to pointwise convergence for such currents.
As an application, the $k$-Hessian measures are calculated by a different method in terms of currents.
\medskip

\noindent{\bf Key words:}  Semi-convex function; subgradient;  Cartesian current; Hessian measure.
\medskip

\noindent{\bf 2010 MR Subject Classification:} 28A75, 49Q15, 52A41, 53C65.
%\medskip

\section{Introduction and main results}

This paper is devoted to the study of some  properties and applications of the  subgradient (or subdifferential) graphs of  semi-convex  functions defined on a convex domain
 $\Omega \subset \mathbb{R}^n$.
 Let $\mathcal{L}^n$ and $\mathcal{H}^n$ be the $n$-dimensional Lebesgue and Hausdorff measure as usual.

Following the Cartesian current theory, Giaquinta-Modica-Sou\v{c}ek \cite{GMS} introduced  a class of functions $u\in L^1(\Omega, \mathbb{R}^n)$, named $\mathcal{A}^1(\Omega, \mathbb{R}^n)$,  such that $u$ is approximately differentiable a.e. and all minors of the Jacobian matrix $Du$ are summable in $\Omega$.
For $u\in \mathcal{A}^1(\Omega,\mathbb{R}^n)$, it is well defined an integer multiplicity (i.m.)
rectifiable current $G_{u}$ carried by the  rectifiable graph of $u$. More precisely,
$$G_{u}=\tau(\mathcal{G}_{u,\Omega},1,\xi_{u}).$$
The unit $n$-vector $\xi_{u}(x,u(x))=\frac{M(Du(x))}{|M(Du(x))|}$ given at each point $(x,u(x))\in \mathcal{G}_{u,\Omega}$ provides an orientation to the approximate tangent space $Tan^n(\mathcal{G}_{u,\Omega},x)$. And the  rectifiable graph of $u$ is given by
$$\mathcal{G}_{u,\Omega}=\{(x,u(x))  ~|~ x \in \mathcal{L}_u\cap A_D(u) \cap \Omega \},$$
where $\mathcal{L}_u$ is the set of Lebesgue points and $A_D(u)$ is the set of approximate differentiability points of $u$,
 for more details see \cite[Vol. I, Sect. 3.2.1]{GMS}.
Moreover the area of  $\mathcal{G}_{u,\Omega}$ is equal to the mass of $G_{u}$, i.e.,
$$\mathcal{H}^n(\Gamma_{u,\Omega})=\int_{\Omega}|M(Du(x))|dx=\mathbf{M}(G_u).$$

In the sequel we study the properties of subgradient $\partial w$ of
a semi-convex function $w$ in terms of Cartesian currents.
The initial motivation of our work is the following: Alberti-Ambrosio \cite{AA} studied  some analytical properties of monotone set-valued maps defined on $\mathbb{R}^n$,
defined $n$-currents on $\mathbb{R}^n\times\mathbb{R}^n$ for maximal monotone maps on $\mathbb{R}^n$ and gave some continuity and approximation results for such currents.
As we know, an important class of maximal monotone maps is represented by the subgradients of convex functions.
A natural problem is raised whether we can extend the definitions and results of current for $\mathcal{A}^1(\Omega, \mathbb{R}^n)$ and maximal monotone maps to the subgradients of semi-convex functions (denoted by $W(\Omega)$).
Here we try to discuss this problem.
 More precisely,
we define an i.m. rectifiable current $G_{\partial w}$ carried by the subgradient graph (denoted by $\Gamma_{\partial w,\Omega}$) of $w$  such that the current has zero boundary and the orientation  in ``nonvertical parts" is consistent with the one given in the class $\mathcal{A}^1(\Omega,\mathbb{R}^n)$ a.e. The following is our first main result.
\begin{theorem} \label{thm1}
If $w \in W(\Omega)$ and  a single-valued map $f:\Omega\rightarrow\mathbb{R}^n$ such that $f(x)\subset \partial w(x)$ for any $x\in \Omega$. Then there exists an i.m. rectifiable current $G_{\partial w}$ such that:
\begin{enumerate}
\item[{\em (\romannumeral1)}]  $G_{\partial w }=\tau(\Gamma_{\partial w,\Omega},1,\xi)\in \mathcal{R}_{n,loc}(\Omega \times \mathbb{R}^n)$.
\item[{\em (\romannumeral2)}]  $\xi(x,f(x))=\frac{M(Df(x))}{|M(Df(x))|}$ for $\mathcal{L}^n~ a.e.~ x\in \Omega$, and $G_{Dw}=G_{\partial w}$ when $w\in C^2$.
\item [{\em (\romannumeral3)}] $\partial G_{\partial w} \llcorner \Omega \times \mathbb{R}^n =0$ and $\mathcal{H}^n(\Gamma_{\partial w, B})=\mathbf{M}_{B \times \mathbb{R}^n}(G_{\partial w})$ for any Borel set $B\subset\subset \Omega$.
\end{enumerate}
\end{theorem}
It turns out that $G_{\partial w}$ is the push-forward of an i.m. rectifiable current under a rotation transformation. However this quantity is well-defined by the fact that the current is  independent of the choice of rotation transformations. For more details, see Theorem \ref{thm32}.

An important problem is to characterize the Cartesian currents $T \in D_n(\Omega \times \mathbb{R}^n)$  for which there is a sequence of smooth maps $u_k:  \Omega \rightarrow \mathbb{R}^n$ such that
$$G_{u_k} \rightharpoonup T,~~~~\mathbf{M}(G_{u_k})\rightarrow \mathbf{M}(T).$$ %
Such a question is connected with the problem of relaxation of the area integral for nonparametric graphs which is discussed by Giaquinta-Modica-Sou\v{c}ek in \cite{GMS2}, \cite [Vol. II , Ch. 6]{GMS}. % which is discussed in \cite [Vol.II,Ch.6.]{GMS}.
Then Mucci made efforts to investigate the problem and showed that functions satisfying some certain conditions can be approximated weakly in the sense of currents and in area by graphs of smooth map (see \cite{MD1,MD2,MD3,MD4}).
 It motivates us to focus on the question
wether $G_{\partial u}$ can be approximated by smooth currents.
In view of this, we show the following weak continuity theorem in current sense for semi-convex functions.

\begin{theorem}[\bf{Weak continuity theorem }]\label{thm2}
 If $w,w_k \in W(\Omega,c)$ such that $w_k \rightarrow w$ pointwise in $\Omega$ as $k \rightarrow \infty$,
then $G_{\partial w_k} \rightharpoonup G_{\partial w}$ in $\mathcal{D}^n(\Omega \times \mathbb{R}^n)$.
\end{theorem}
 According to  Theorem  \ref{thm2},  the current $G_{\partial w}$ carried by the graph of subgradient of a semi-convex function $\omega$ is Lagrangian, which one can think of as meaning ``weakly curl-free" (see \cite{J} and below for more details).  Moreover, the current  can be weakly approximated by smooth currents if $\omega$ is also  Lipschitz.

As an application of the current defined as above,
a method is proposed to calculate the $k$-Hessian measures for semi-convex functions.
 Trudinger-Wang \cite{TW1,TW2} introduced the notion of $k$-Hessian measures as  Borel measures
associated to $k$-convex functions to study weak solutions of some elliptic partial differential equations.
Then   Colesanti-Hug \cite{AD} proved
that the $k$-Hessian measures of  semi-convex functions can be defined as coefficients of a  local Steiner type formula,
and pointed out  the equivalence of the two definitions in the class of semi-convex functions.
Now, a formula for the $k$-Hessian measures in terms of currents for semi-convex functions is as follow.
\begin{theorem}\label{thm4}
 Let $w \in W(\Omega)$ and  $G_{\partial w}:=\tau (\Gamma_{\partial w,\Omega},1, \xi)$,
 then for every Borel subset $B \subset \subset \Omega$,
 the $k$-Hessian measure of $w$ can be written as
 $$F_k(w,B)= \sum_{|\alpha|=k} \sigma(\alpha,\overline{\alpha}) \int _{\Gamma(\partial w,B)} \xi^{\alpha \overline{\alpha} }(x,y) d\mathcal{H}^n (x,y)~~~~k=0,1,\cdot\cdot\cdot,n.$$
  In particular, $F_0(w,B)= \mathcal{L}^n(B); F_n(w,B)=\mathcal{L}^n (\partial w(B))$ if $w$ is convex.
 \end{theorem}
This paper is organized as follows. Some facts and notion about semi-convex functions, set-valued maps and Cartesian currents are given in Section 2.
Then we prove Theorem \ref{thm1}  in Section 3.
In Section 4 we show the weak continuity theorems  for semi-convex functions.
Finally in Section 5 we give a different formula for $k$-Hessian measures of semi-convex functions.

\section{Preliminaries}
This section reviews some notion and basic facts about  semi-convex functions, set-valued maps and Cartesian currents.
For more details, see \cite{HJL,F,JPA,J,GMS,RTR}.
\begin{definition}
 A real-valued function $ w:\Omega \longrightarrow \mathbb{R}$ is called semi-convex if there exists $c \geq 0$ such that the function $w(x)+\frac{c}{2}|x|^2 $ is convex in $\Omega$.

 For a    semi-convex function $w$  in  $\Omega$, the semiconvexity modulus of $w$ is defined by
 $$sc(w,\Omega):=\inf\{c \mid w+\frac{c}{2} |x|^2 ~\mbox{is convex in}~\Omega\}.$$
 We set the class of semi-convex functions by  $W(\Omega)$  and  $W(\Omega,c):= \{w \in W(\Omega) \mid sc(w,\Omega) \leq c\}$.
 \end{definition}
 \begin{examples}
 \begin{enumerate}
\item[(\romannumeral1)]  The simplest examples of semi-convex functions are convex functions.
\item [(\romannumeral2)]  The viscosity solutions of some   Hamiltonian-Jacobian equations are semi-concave, the class of functions $u$ such that $-u\in W(\Omega)$, see Lions \cite{L}.
\item [(\romannumeral3)] The distance function from a closed subset $K$ of a $n$-dimensional Riemannian manifold $(M,g)$ are locally semi-concave in $M\setminus K$, see Mantegazza-Mennucci\cite{MM}.
\end{enumerate}
 \end{examples}
We denote by $P(\mathbb{R}^n)$  the collection of subsets of $\mathbb{R}^n$, $P_0(\mathbb{R}^n):=P(\mathbb{R}^n)\setminus \{\emptyset\}$ and $I$ both the identity map on $\mathbb{R}^n$ and the $n \times n$ identity matrix.
Given a set-valued map $F: \Omega \rightarrow P(\mathbb{R}^n)$ and $A\subset \mathbb{R}^n$, we set
\begin{align*}
graph ~of~ F,~ &\Gamma_{F,A}:=\{(x,y)\in\mathbb{R}^{n}\times \mathbb{R}^{n} \mid y\in F(x), x\in \Omega\}, \\
image ~of~ F,~ &F(A):=\{y\in \mathbb{R}^n \mid y\in F(x), x\in A\}.\\
\end{align*}
 %---------------------------------------------------------------------------------------%
\begin{definition}
A set-valued map $F: \Omega \rightarrow P(\mathbb{R}^n)$ is  monotone in $\Omega$ if its graph is monotone, i.e.,
$$\langle y_1-y_2, x_1-x_2 \rangle\geq 0$$
for all $(x_i,y_i)\in \Gamma_{F,\Omega}$ with $i=1,2$. A monotone map $F$ is  maximal in $\Omega$ if there is no other  monotone set-valued map in  $\Omega$ whose graph strictly contains the graph of $F$.
\end{definition}

%--------------------------------------------------------------------------------------%
 \begin{definition}
Let $ w:\Omega \longrightarrow \mathbb{R}$, the subgradient (subdifferential) $\partial w(x)$ of $w$ at $x$ is defined by
$$\partial w(x)=\{p \in \mathbb{R}^n \mid  \liminf_{y \rightarrow x}\frac{w(y)- w(x)-\langle p,y-x\rangle}{|y-x|} \geq 0\}.$$
\end{definition}

%---------------------------------------------------------------------------------------%

We also recall some elementary properties of subgradient of semi-convex functions.
\begin{proposition}\label{pro2}
Let $ w\in W(\Omega,c)$.
\begin{enumerate}
\item[{\em(\romannumeral1)}]  $w$ is locally Lipschitz in $\Omega$ and the image $\partial w(A)$ is bounded for any bounded subset $A\subset\subset\Omega$.
\item [{\em(\romannumeral2)}]  $\partial w(x)$ is a nonempty, closed and convex set for any $x \in \Omega$. Moreover,
    $$y \in \partial w(x) \Leftrightarrow w(z)\geq w(x)+ (z-x)\cdot y-\frac{c}{2} |z-x|^2~~ \mbox{for all}~ z \in \Omega.$$
\item [{\em(\romannumeral3)}]  $w$ has a second derivative for $\mathcal{L}^n$ a.e. on $\Omega$. Moreover, $\partial w$ is a maximal semi-monotone map in $\Omega$, i.e., $\partial w+c I$ is maximal monotone map in $\Omega$.
\item [{\em(\romannumeral4)}] $(sI+\partial w)(\Omega)=\mathbb{R}^n$ if $\Omega=\mathbb{R}^n$, where $s>c$.
\end{enumerate}
\end{proposition}
\begin{proof}
If we take $v(x)=w(x)+\frac{c}{2} |x|^2$, then clearly $v$ is convex.
Now (\romannumeral1), (\romannumeral2), (\romannumeral3) and (\romannumeral4) are immediately followed from \cite[Proposition 2.1.5, 2.4.1 ]{AP}, \cite[Proposition 2.1]{AA2}, \cite[Theorem 3.2]{AA} and \cite[Theorem 3.5.8]{AD}, respectively.
\end{proof}

In order to obtain the main results, we introduce the definition of geometrically derivatives for set-valued maps from the choice of tangent cones to the graphs (see \cite{JPA}).
\begin{definition}
Let  a set-valued map $F: \Omega\rightarrow P(\mathbb{R}^n)$.
The contingent derivative $DF(x,y)$ of $F$ at $(x,y)\in \Gamma_{F,\Omega}$ is the set-valued map from $\mathbb{R}^n$ to ${\mathbb{R}^n}$ defined by
$$\Gamma_{DF(x,y),\mathbb{R}^n}:=\left \{(p,q) \mid \liminf_{h\rightarrow 0^+}\frac{d( (x+hp,y+hq) ,\Gamma(F,\Omega))}{h}=0\right\}.
$$
\end{definition}

It is very convenient to have the following characterization of contingent derivatives in terms of sequences:
$q \in DF(x,y)(p)$ if and only if there exist $ h_m \rightarrow 0^+$, $p_m\rightarrow p$ and $q_m\rightarrow q$ as $m\rightarrow \infty$ such that $y+h_m q_m \in F(x+h_m p_m).$

\begin{proposition}
If  $F:=f$  is a single-valued map and differentiable at  $x$, then $Df(x,f(x))(p)=Df(x)p$ for any $p\in \mathbb{R}^n$.
\end{proposition}
\begin{proof}
The proof can be seen in \cite[Proposition 5.1.2]{JPA}.
\end{proof}
Next we recall some notation and facts about currents and Geometric Measure Theory.

For integers  $n,N \geq 2$, we shall use the standard notation for ordered multi-indices
$$I(k,n):=\{\alpha=(\alpha_1,\cdot\cdot\cdot,\alpha_k) \mid \alpha_i  ~\mbox{integers}, 1\leq \alpha_1 <\cdot\cdot\cdot< \alpha_k\leq n\}.$$
 Set $I(0,n)=\{0\}$ and $|\alpha|=k$ if $\alpha \in I(k,n)$. If $\alpha\in I(k,n)$, $k=0,1,\cdot\cdot\cdot,n$,  $\overline{\alpha}$ is the element in $I(n-k,n)$ which complements $\alpha$  in $\{1,2,\cdot\cdot\cdot,n\}$ in the natural increasing order. So $\overline{0}=\{1,2,\cdot\cdot\cdot,n\}$. For $i \in \alpha$, $\alpha-i$ means the multi-index of length $k-1$ obtained by removing $i$ from $\alpha$.

Let $A=(a_{ij})_{n \times n}$ and $B=(b_{ij})_{n \times n}$ be  $n \times n$ matrixes.
Given two ordered multi-indices with $ \beta \in I(k,n), \alpha\in I (n-k,n) $, then
$A_{\overline{\alpha}}^{\beta}$ denotes
the $k \times k $-submatrix of $A$ with rows $(\beta_1,\cdot\cdot\cdot,\beta_k)$ and columns $(\overline{\alpha}_1,\cdot\cdot\cdot,\overline{\alpha}_k)$. Its determinant will be denoted by
$$M_{\overline{\alpha}}^{\beta}(A):=\det A_{\overline{\alpha}}^{\beta}.$$
We shall set
$$M_{\alpha \beta}(A,B):=\mbox{det}(C),~~\mbox{where}~ c_{ij}=\begin{cases}
a_{ij},~~~~i\in \alpha, \\
b_{ij},~~~~i \in \beta.
\end{cases}$$

The adjoint of $A_{\overline{\alpha}}^{\beta}$ is  defined  by the formula
$$(\adj A_{\overline{\alpha}}^{\beta})_j^i:= \sigma(i,\beta-i) \sigma(j,\overline{\alpha}-j) \det A_{\overline{\alpha}-j}^{\beta-i}~~~~ i \in \beta, j\in \overline{\alpha},$$
where $\sigma(\cdot,\cdot)$  is the sign of the permutation with reorders. So Laplace  formulas can be written as
$$M_{\overline{\alpha}}^{\beta}(A)= \sum_{j \in \overline{\alpha}} a_{ij} (\adj A_{\overline{\alpha}}^{\beta})_j^i.$$

Let $U$ be a open set in $\mathbb{R}^n$, we denote by $\mathcal{D}^k(U)$ the spaces of compactly supported $k$-form in $U$.  The dual space to $\mathcal{D}^k(U)$ is the class of $k$-currents $\mathcal{D}_k(U)$. For any open set $V\subset\subset U$ the mass of a current $T\in \mathcal{D}_k(U)$ in $V$ is
$$\mathbf{M}_{V}(T):=\sup\{T(\omega) \mid \omega \in \mathcal{D}^k(U),~\mbox{spt}~ \omega \subset V, \| \omega\|\leq 1 \},$$
 and $\mathbf{M}(T):=\mathbf{M}_U(T)$ denote the mass of $T$.

 A current $T=\tau(\mathcal{M}, \theta, \xi) \in \mathcal{D}_k(U)$, is called an  integer multiplicity rectifiable $k$-current (briefly i.m. rectifiable current) if it can be expressed
$$T(\omega)=\int_{\mathcal{M}}\langle \omega(x), \xi(x)\rangle \theta(x) d\mathcal{H}^k(x),~~~~\omega \in \mathcal{D}^k(U),$$
where $\mathcal{M}$ is an $\mathcal{H}^k$-measurable countably $k$-rectifiable subset of $U$,
$\theta$ is a locally $\mathcal{H}^k$-measurable positive integer-valued function,
and $\xi: \mathcal{M}\rightarrow \wedge_k(\mathbb{R}^{n})$ is a $\mathcal{H}^k$-measurable function such that for $\mathcal{H}^n$-a.e. point $x\in\mathcal{M}$,
$\xi(x)$ provides an orientation to the approximate tangent spaces  $Tan^k(\mathcal{M}, x)$.
  $\theta$ is called the multiplicity and $\xi$ is called the orientation for $T$.
  The i.m. rectifiable  $k$-currents in  $\mathcal{D}_k(U)$ is  denote by $\mathcal{R}_k(U)$ if $T$ has finite mass, and $T \in \mathcal{R}_{k,loc}(U)$ if $T$ has local finite mass.

  Let  $T=\tau (\mathcal{M},\theta,\xi) \in \mathcal{R}_k(U)$, and   $f: U\rightarrow V \subset \mathbb{R}^n$  be a Lipschitz map such that  $f_{| \mbox{spt}~T} $  is proper. Then the push-forward of $T$ under $f$  turns out to be an i.m rectifiable  $k$-current which can be explicitly written as (see \cite{F,GMS})
\begin{equation}\label{formula1}
\begin{split}
f_{\sharp} T(\omega) &= \int_{\mathcal{M}}\langle w(f(x))),(\wedge_k d^{\mathcal{M}}f)\xi(x)\rangle \theta(x) d\mathcal{H}^k(x)\\
                           &=\int_{f(\mathcal{M})}\langle\omega(y), \sum_{x\in f^{-1}(y) \cap \mathcal{M}_{+}} \theta (x) \frac{(\wedge_k d^{\mathcal{M}}f)\xi(x)}{|(\wedge_k d^{\mathcal{M}}f)\xi(x)|}\rangle d\mathcal{H}^k (y),\\
\end{split}
\end{equation}
where
$$\mathcal{M}_+=\{x \in \mathcal{M} \mid J_f^{\mathcal{M}}(x)=|(\wedge_k d^{\mathcal{M}}f)\xi(x)|>0\}.$$

$T=\tau(\mathcal{M}, \theta, \xi) \in \mathcal{R}_{n,loc}(\Omega\times \mathbb{R}^n)$ ($n\geq2$)  is called Lagrangian if for $\mathcal{H}^n$ a.e. $(x,y) \in \mathcal{M}$, the approximate tangent space $Tan^n(\mathcal{M}, (x,y))$ satisfies
 $$\langle\phi, \tau_1 \wedge\tau_2 \rangle=0 ~\mbox{for any two vectors}~\tau_1,\tau_2~\mbox{tangent to}~Tan^k(\mathcal{M}, (x,y))$$
where $\phi:=\sum_{i=1}^{\infty}dx^i \wedge dy^i$. One can check that $T$ is Lagrangian if and only if
$$T(\phi \wedge \eta)=0 ~~\mbox{for any}~ \eta \in \mathcal{D}_{n-2}(\Omega\times \mathbb{R}^n).$$

\section{The proof of Theorem 1.1 }
For $u \in \mathcal{A}^1(\Omega,\mathbb{R}^N)$,
 the  current $G_{u}=\tau(\mathcal{G}_{u,\Omega},1,\xi_{u}) \in \mathcal{R}_n(\Omega \times \mathbb{R}^N)$ is defined for $\omega \in \mathcal{D}^n(\Omega \times \mathbb{R}^N)$ by (see \cite[Vol. I, Sect. 3.2.1]{GMS})
\begin{align*}
  G_{u}(\omega) &= \int_{\mathcal{G}_{u,\Omega}} \langle\omega,\xi_{u}\rangle d \mathcal{H}^n \\
                   &= \int_{\Omega} \langle\omega(x,u(x)),M(Du(x))\rangle dx \\
                   &= \sum_{|\alpha|+|\beta|=n}  \sigma(\alpha,\overline{\alpha}) \int_{\Omega} \omega_{\alpha\beta}(x,u(x)) M_{\overline{\alpha}}^{\beta} (Du(x)) dx,
\end{align*}
where $M(Du)$ is the $n$-vector in $\bigwedge_n(\mathbb{R}^{n+N})$ given by
$$M(Du)=(e_1+\sum_{i=1}^{N} D_1 u^i \epsilon_i)\wedge ...\wedge(e_n+\sum_{i=1}^{N} D_n u^i \epsilon_i),$$
$\{e_i\}_{i=1}^n$, $\{\epsilon_i\}_{i=1}^{n}$ being canonical basis for $\mathbb{R}^n_x$, and $\mathbb{R}^n_y$, respectively.

In order to prove the main result, some lemmas are introduced as follows.
\begin{lemma}\label{lem13}
Let  $F: \Omega\rightarrow P(\mathbb{R}^n)$ be a set-valued map. If there exists a constant $s$ such that
$$\langle y_1-y_2, x_1-x_2\rangle \geq s|x_1-x_2|^2$$
for all $(x_1,y_1), (x_2,y_2) \in \Gamma_{F,\Omega} $,
then the contingent derivative $DF(x,y)$ at every point  $(x,y) \in \Gamma_{F,\Omega}$ is positive definite in the sense that
$$p\cdot q\geq s |p|^2$$
for all  $(p,q) \in \Gamma_{DF(x,y),\mathbb{R}^n}$.
\end{lemma}
\begin{proof}
Given $(x,y)\in \Gamma_{F,\Omega}, (p,q) \in \Gamma_{DF(x,y)}$, there exist $h_k \rightarrow 0^+$, $p_k\rightarrow p$ and $q_k\rightarrow q$ such that
$$(x,y)+h_k(p_k, q_k) \in \Gamma_{F,\Omega}.$$
So $y+h_k q_k \in  F (x+h_k p_k)$ and
$$h_k^2 p_k\cdot q_k =\langle  y+h_k q_k -y, x+h_k p_k-x\rangle\geq s h_k^2 |p_k|^2,$$
which implies that $p\cdot q\geq s|p|^2$.
\end{proof}

\begin{lemma}\label{lem1}
Let  $A$   be an   $n\times n$  matrix and semi ($s=0$), weak ($s<0$), strongly ($s>0$) positive definite in the sense that
$$ x^{T} A x \geq s x^{T} x$$
for all $x \in \mathbb{R}^n ,x \neq 0$. Then  $Re(\lambda)\geq s$,
where $Re(\lambda)$ is the real part of any eigenvalue  $\lambda$  of  $A$. Furthermore,
$\det A\geq ~(\mbox{or}>)~0$  if  $A$ is semi (or strongly) positive definite.
\end{lemma}

\begin{proof}
Let  $\lambda=\mu+ i\nu \in \mathbb{C}$  be an eigenvalue of  $A$ with  $\mu,\nu \in \mathbb{R}$, and  $z \in\mathbb{C}^n$  be a right eigenvector associated with  $\lambda$. Decompose  $z$  as  $x+iy$  with  $x,y \in \mathbb{R}^n$. Then $(A-\lambda)z=0$, and thus
$$\begin{cases}
(A-\mu)x+\nu y=0, \\
(A-\mu)y-\nu x=0,\\
\end{cases}$$
which implies that $x^{T}(A-\mu)x+y^{T}(A-\mu)y=\nu(y^{T}x-x^{T}y)=0$, and hence
$$\mu=\frac{x^{T}Ax+y^{T}Ay}{x^{T}x+y^{T}y} \geq s.$$
\end{proof}

%---------------------------------------------------------------------------------------%

%---------------------------------------------------------------------------------------%
\begin{lemma}\label{lem22}
Let  $A=(a_{ij})_{n \times n}$  be an  $n\times n$ matrix, then for any  $ c,d>0$
$$ \sum_{|\alpha|+|\beta|=n}(M_{\alpha \beta}(A,cI-dA))^2>0.$$

\end{lemma}

\begin{proof}
If  $A=0$, then  $M_{\alpha\beta}(A,cI-A)\neq0$ for $\alpha=0, \beta=(1,2,...n)$.
If  $rank(A)=r>0$, then there exist  $\alpha,\overline{\beta}  \in I(r,n)$ such that $M^{\alpha}_{\overline{\beta}}(A) \neq 0$. $A$ can be written as   $A:=(\rho_1,\rho_2,\cdot\cdot\cdot,\rho_n)^{T}$, where $ \rho_i  \in \mathbb{R}^n$. Then  $\rho_{\alpha_1}, \rho_{\alpha_2},\cdot\cdot\cdot,\rho_{\alpha_r}$ are linearly independent and form a  basis of  $\rho_1, \rho_2, \cdot\cdot\cdot,\rho_n$. Let   $\longrightarrow$ be reversible
linear transformations, then some tedious manipulation yields
$$\left(\rho_{\alpha_1}, \cdot \cdot\cdot, \rho_{\alpha_r}, cI_{\beta_{1}}-d\rho_{\beta_1}
, \cdot \cdot \cdot, cI_{\beta_{n-r}}-d\rho_{\beta_{n-r}}\right)^{T}
\longrightarrow
\left(
  \begin{array}{cc}
    A^{\alpha}_{\overline{\beta}} & 0 \\
    0 & I_{\beta} \\
  \end{array}
\right),$$
which implies  $M_{\alpha \beta}(A,cI-A) \neq 0$.
\end{proof}
%---------------------------------------------------------------------------------------%

  Let $F: \Omega\rightarrow P_0(\mathbb{R}^n)$ be a set-valued map such that $F+cI$ is maximal monotone in $\Omega$, then $f:\Omega\rightarrow\mathbb{R}^n$ such that $f(x)\subset F(x)$ for any $x\in \Omega$. According to \cite[Theorem 3.2]{AA}, $f$ is  approximately differentiable a.e.  with approximate differential $Df$. Given $s>c$, we define a rotation transformation $\Theta$ on $\mathbb{R}^n \times \mathbb{R}^n$ by
$$(x',y') \mapsto (\cos \theta x'-\sin \theta y', \sin \theta x'+ \cos \theta y'),$$
where $\theta:= \arccos \frac{2s}{\sqrt{1+4s^2}}< \arccos \frac{2c}{\sqrt{1+4c^2}}:=\theta_0$. With the help of the preceding lemmas we can now prove the following theorem.

\begin{theorem} \label{thm32}
  $F$, $f$ and $\Theta$ are given as above.  The following hold:
\begin{enumerate}
\item[{\em (\romannumeral1)}]  $\Theta^{-1}(\Gamma_{F,\Omega})$  can be written as the graph of a Lipschitz function  $u$, i.e.,
    $$\Gamma_{F,\Omega} = \Theta(\Gamma_{u,D}),$$
where $D:=\{\cos \theta x+ \sin \theta y \mid (x,y) \in \Gamma_{F,\Omega}\}$.
\item [{\em (\romannumeral2)}] $\Theta_{\sharp} G_u  = \tau (\Gamma_{F,\Omega},1,\xi)$.
\item [{\em(\romannumeral3)}]   $\xi(x,f(x))=\xi_f(x, f(x))$ for $\mathcal{L}^n~ a.e.~ x\in \Omega$,
where $\xi_f(x,f(x)):= M(Df(x))/|M(Df(x))|$.
\item [{\em(\romannumeral4)}] If  $0<\theta_1<\theta_2<\theta_0$, then  $\Theta_{1 \sharp} G_{u_{1}}  = \Theta_{2 \sharp} G_{u_{2}} $.
\end{enumerate}
\end{theorem}

\begin{proof}
It is simple to show that there exists $l>0$ such that
$$|y'_1-y'_2|^2\leq l |x'_1-x'_2 |^2~~~\mbox{for any}~ (x'_1,y'_1), (x'_2,y'_2) \in \Theta^{-1}(\Gamma_{F,\Omega}),$$
where $l\geq \max\{4s^2,\frac{4sc+1}{4s^2-4sc}\}$, which implies (\romannumeral1).
Moreover, we can show that  $D$ is a domain in $\mathbb{R}^n$ and
$$\mbox{Lip}(u,D):= \sup \left\{ \frac{|u(x_1)-u(x_2)|}{ |x_1-x_2|} \mid x_1,x_2 \in D, x_1\neq x_2  \right\}\leq \max\left \{2s,\sqrt{\frac{4sc+1}{4s^2-4sc}}\right\}.$$
 Then the i.m. rectifiable current  $G_{u} =\tau(\Gamma_{u,D},1,\xi_u)$ is carried by the graph of $u$.

An easy deduction gives that
\begin{equation}\label{formula2}
\langle u(x'_1)-u(x'_2),x'_1-x'_2\rangle\leq 2s |x'_1-x'_2 |^2
~\mbox{for any}~ x'_1,x'_2 \in D.
\end{equation}

%{\bf Claim 3:}  $\Theta_{\sharp} G_u  = \tau (\Gamma_{F,\Omega},1,\xi) \in \mathcal{R}_n(\Omega \times \mathbb{R}^n)$.

For any $x'\in D$, let $\mathcal{M}:=\Gamma_{u,D}$ and $A:=(a_{ij})_{n \times n}= Du(x')$, then

$$M(Du(x'))=(e'_1+\sum_{s=1}^n a_{s1} \epsilon'_s) \wedge\cdot\cdot\cdot\wedge (e'_n+\sum_{s=1}^n a_{sn} \epsilon'_s).$$
Let $\tau_i:=e'_i+\sum_{s=1}^n a_{si} \epsilon'_s$ and $\zeta_i:=d^{\mathcal{M}}\Theta(\tau_i)$. Note that

\begin{align*}
\zeta_i&= \sum_{s=1}^n(\tau_i \cdot \nabla^{\mathcal{M}} \Theta^s) e_s+ \sum_{s=1}^n (\tau_i \cdot \nabla^{\mathcal{M}} \Theta^{n+s}) \varepsilon_s\\
                           &=\sum_{j=1}^n (\cos \theta \delta_{ji}- \sin \theta a_{ji})e_j+ \sum_{j=1}^n ( \sin \theta \delta_{ji}+\cos \theta a_{ji}) \varepsilon_j.
\end{align*}
Set  $P=(\cos \theta I- \sin \theta A)$ and $Q=(\sin \theta I+ \cos \theta A)$.
Then  $Q=\sqrt{(1+4c^2)}I-2cP$ and
\begin{align*}
\zeta_1 \wedge \cdot \cdot \cdot \wedge \zeta_n &=(\sum_{j=1}^n p_{j1}e_j+ \sum_{j=1}^n q_{j1} \varepsilon_j) \wedge \cdot \cdot \cdot \wedge (\sum_{j=1}^n p_{jn}e_j+ \sum_{j=1}^n q_{jn} \varepsilon_j)\\
                           &=\sum_{|\alpha|+|\beta|=n} \sum_{|\gamma|=|\beta|}  \sigma(\overline{\gamma},\gamma) M^{\alpha}_{\overline{\gamma}}(P)M_{\gamma}^{\beta} (Q) e_{\alpha} \wedge \varepsilon_{\beta}\\
                           &=\sum_{|\alpha|+|\beta|=n} \sigma(\overline{\beta},\beta) M_{\alpha \beta}(P,Q)e_{\alpha} \wedge \varepsilon_{\beta}.
\end{align*}
By Lemma \ref{lem22}
\begin{align*}
\mathcal{M}_+ &=\{(x',y') \in\mathcal{M} \mid~  |(\wedge_n d^{\mathcal{M}}\Theta)(M(Du(x')))| >0\}\\  &=\{(x',u(x')) \in \mathcal{M} \mid ~ |\zeta_1 \wedge \cdot\cdot\cdot\wedge \zeta_n| >0\}            \\
               &=\{(x',u(x')) \in \mathcal{M} \mid ~ Du(x') ~\mbox{exist}\},
\end{align*}
 which implies that  $\mathcal{H}^n(\mathcal{M}\backslash \mathcal{M}^+) =0$.
According to (\ref{formula1}), it follows that for any  $\omega(x,y) \in \mathcal{D}^n(\Omega \times \mathbb{R}^n)$,
\begin{align*}
\Theta_{\sharp} G_u  (\omega(x,y))&=\int_{\Gamma_{F,\Omega}}\big\langle\omega(x,y),  \frac{\zeta_1 \wedge \cdot\cdot\cdot\wedge \zeta_n}{|(\zeta_1 \wedge \cdot\cdot\cdot\wedge \zeta_n|}(\cos \theta x+\sin \theta y,-\sin \theta x+\cos \theta y)\big\rangle d\mathcal{H}^n (x,y)\\
                       &=\tau(\Gamma_{F,\Omega},1,\xi)(\omega(x,y)),
\end{align*}
where the  orientation $\xi(x,y)=\frac{\zeta_1 \wedge \cdot\cdot\cdot\wedge \zeta_n}{|(\zeta_1 \wedge \cdot\cdot\cdot\wedge \zeta_n|}(\cos \theta x+\sin \theta y,-\sin \theta x+\cos \theta y)$
for $\mathcal{H}^n$-a.e. $(x,y) \in \Gamma_{F,\Omega}$.
Therefore  $\Theta_{\sharp} G_u  =\tau(\Gamma_{F,\Omega},1,\xi)$.

  Set $E:= \{x \in \Omega \mid x \in \mathcal{L}_f, Df(x), Du(\cos \theta x+\sin \theta f(x))~\mbox{exists}\}$.
According to Proposition \ref{pro2} (\romannumeral3) and the fact that $u \in L(D)$ , it follows that $\mathcal{L}^n(\Omega\backslash E)=0$.

%{\bf Claim 4:}    $G_{f} \subset \Theta_{\sharp} G_u$.

Fix  $x_0 \in E$ and  denote $A:=Du(\cos \theta x_0+\sin \theta f(x_0))$, $B:=Df(x_0)$. Since $u
 (\cos \theta x_0+\sin \theta f(x_0))=-\sin \theta x_0+\cos \theta f(x_0)$, then
 $$A(\cos \theta I+ \sin \theta B)=-\sin \theta I+\cos \theta B,~~(\cos \theta A+\sin \theta I)=(\cos \theta I-\sin \theta A)B.$$
 So
 $$(\cos \theta I- \sin \theta A)(\cos \theta I+\sin \theta B)=I,$$
which implies that  $(\cos \theta I- \sin \theta A)$ is reversible
and  $(\sin \theta I+ \cos \theta A)(\cos \theta I- \sin \theta A)^{-1}=B$.
Let  $P=(\cos \theta I- \sin \theta A)$ and $Q=(\sin \theta I+ \cos \theta A)$ for convenience, then
\begin{align*}
&\ \ \ \ \zeta_1 \wedge \cdot \cdot \cdot \wedge \zeta_n (\cos \theta x_0+\sin \theta f(x_0),-\sin \theta x_0+\cos \theta f(x_0))\\
&=(\sum_{j=1}^n p_{j1}e_j+ \sum_{j=1}^n q_{j1} \varepsilon_j) \wedge \cdot \cdot \cdot \wedge (\sum_{j=1}^n p_{jn}e_j+ \sum_{j=1}^n q_{jn} \varepsilon_j)\\
                           &=(\sum_{j=1}^n p_{j1}e_j+ \sum_{j=1}^n (BP)_{j1} \varepsilon_j) \wedge \cdot \cdot \cdot \wedge (\sum_{j=1}^n p_{jn}e_j+ \sum_{j=1}^n (BP)_{jn} \varepsilon_j)\\
   &= \det(P^{T})(e_1+ \sum_{s=1}^n b_{s1} \varepsilon_s) \wedge \cdot \cdot \cdot \wedge (e_n+ \sum_{s=1}^n b_{sn} \varepsilon_s).
\end{align*}
 Note that $\det P\geq0$ by (\ref{formula2}), Lemma \ref{lem13} and Lemma \ref{lem1}, and thus
\begin{align*}
\xi(x_0,f(x_0))&=\s(P) \frac{(e_1+ \sum_{s=1}^n b_{s1} \varepsilon_s) \wedge \cdot \cdot \cdot \wedge (e_n+ \sum_{s=1}^n b_{sn} \varepsilon_s)}{|(e_1+ \sum_{s=1}^n b_{s1} \varepsilon_s) \wedge \cdot \cdot \cdot \wedge (e_n+ \sum_{s=1}^n b_{sn} \varepsilon_s)|}\\
&=M(Df(x_0))/|M(Df(x_0))|.
\end{align*}
Therefore  $\xi_f(x, f(x))=\xi(x,f(x))$ for $\mathcal{L}^n~ a.e. x\in \Omega$.

  If  $0<\theta_1<\theta_2<\theta_0$ with $\theta_i:= \arccos \frac{2s_i}{\sqrt{1+4s_i^2}}$, i=1,2.  Then there exists a  transformation $\Theta_{3}$ such that $\Theta_{3}(\Gamma_{u_{2},D_2}) = \Gamma_{u_{1},D_1}$. In order not to confuse matters, we write
$$\Theta_1:\begin{cases}
x=\cos \theta_1 x'- \sin \theta_1 y' \\
y=\sin \theta_1 x' +\cos \theta_1 y',
\end{cases}
\Theta_2:\begin{cases}
x=\cos \theta_2 x''- \sin \theta_2 y'' \\
y=\sin \theta_2 x'' +\cos \theta_i y'',
\end{cases}
\Theta_3:\begin{cases}
x'=\cos \theta_3 x''- \sin \theta_3 y'' \\
y'=\sin \theta_3 x'' +\cos \theta_3 y''.
\end{cases}$$
  Clearly, $\Theta_2= \Theta_1\circ \Theta_3$ and $\theta_2=\theta_1+\theta_3$.

 Let $H(x'')=(\cos \theta_3 I- \sin \theta_3 u_{2})(x'')$ where $x''\in D_{2}$. Some tedious manipulation yields that  there exists $l>0$ such that
 \begin{equation}\label{formlula03}
 \langle H(x''_1)-H(x''_2) ,x''_1-x''_2\rangle\geq l|x''_1-x''_2|^2
 \end{equation}
for any $x''_1, x''_2 \in D_2$.
 For any $x''_0$ such that $Du_{2}(x''_0)$ exists,
it follows from Lemma \ref{lem13} and (\ref{formlula03}) that
$$p^{T}DH(x''_0, H(x''_0))p = p^{T} (\cos \theta_3 I- \sin \theta_3 Du_{2}(x''_0)) p \geq l p^{T}p$$
for all $p \in \mathbb{R}^n$, which implies $\det(\cos \theta_3 I- \sin \theta_3 Du_{2}(x''_0))>0$.
Then an argument similar to the one as above shows that
$$G_{u_{1}} = \Theta_{3\sharp}G_{u_{2}}.$$
Hence
$$\Theta_{2\sharp }G_{u_{2}}=\Theta_{1 \sharp} \circ \Theta_{3\sharp }G_{u_{2}}=\Theta_{1\sharp }G_{u_{1}},$$
which completes the proof.
\end{proof}

\begin{definition}\label{DefG}
Let $F: \Omega\rightarrow P_0(\mathbb{R}^n)$ be a maximal semi-monotone map, we define the Cartesian current $G_F$ associated to $F$ as
$$G_{F}:=\Theta_{ \sharp} G_{u},$$
where $\Theta$, $u$ are given in Theorem \ref{thm32}.
\end{definition}
This quantity is well-defined since  $G_F$ is independent of the rotation transformations  and
the orientation of the current is consistent with the one defined in the class $\mathcal{A}^1(\Omega,\mathbb{R}^n)$.

\begin{proof}[\bf Proof of Theorem 1.1]
The integer multiplicity rectifiable  $n$-current $G_{\partial w}$ carried by $\Gamma_{\partial w,\Omega}$ is defined by Definition \ref{DefG}. For any open set $\Omega'\subset\subset\Omega$, $D':=\pi\circ \Theta(\Gamma_{\partial w,\Omega'})$ is bounded since $\partial w(\Omega')$ is bounded. Then a tedious computation implies that
 \begin{equation}\label{formlula04}
\mathcal{H}^n(\Gamma_{\partial w,\Omega'})=\mathbf{M}_{\Omega'\times\mathbb{R}^n}(G_{\partial w})=
\mathbf{M}_{D'\times\mathbb{R}^n}(G_u)=\int_{D'}|M(Du)|dx'.
 \end{equation}
So (\romannumeral1), (\romannumeral2) can be easily proved by Theorem \ref{thm32} and (\ref{formlula04}).

Given $ \eta (x,y) \in \mathcal{D}^{n-1}(\Omega \times \mathbb{R}^n)$, let $\Theta(D \times \mathbb{R}^n)=U$,
then $U \cap(\Omega \times \mathbb{R}^n)$ is an open set in $\Omega \times \mathbb{R}^n$.
And thus $\mbox{spt}~\eta \cap \mbox{spt}~G_{\partial w}$ is compact in $U \cap (\Omega \times \mathbb{R}^n)$.
So there exists $\zeta \in C_0^{\infty}(U \cap (\Omega \times \mathbb{R}^n))$ such that
$\zeta=1$ in a neighbourhood of $\mbox{spt}~\eta \cap \mbox{spt}~ G_{\partial w}$.
Thus
$$\partial G_{\partial w}(\eta)= G_{\partial w}(d\eta) =G_{\partial w}(\zeta d\eta)=G_{\partial w}(d(\zeta\eta))=\partial G_{u}(\Theta^{\sharp}(\zeta \eta)) =0,$$
where the last equality follows from $\Theta^{\sharp}(\zeta \eta) \in \mathcal{D}^{n-1}(D \times \mathbb{R}^n)$.
So $\partial G_{\partial w} \llcorner \Omega \times \mathbb{R}^n =0$.
\end{proof}

\section{The proof of Theorem 1.2}

\begin{lemma}\label{lma41}
Let $u,  \{u_k\}_{k=1}^{\infty} \subset L(\Omega,\mathbb{R}^n)$ such that $Lip(u_k)$, $Lip(u)$ uniformly bounded and $u_k$ converge uniformly to $u$ in $\Omega$. Then
$$M_{\overline{\alpha}}^{\beta} (Du_k) \rightharpoonup M_{\overline{\alpha}}^{\beta} (Du)~~\mbox{in}~L^1(\Omega)$$
for any ordered multi-indices $\alpha, \beta$ with $|\alpha|+|\beta|=n$.
\end{lemma}
\begin{proof}
Note that Laplace's formulas yield
\begin{align*}
M_{\overline{\alpha}}^{\beta}(Du) &= \sum_{j \in \overline{\alpha}} D_ju^i ((\adj Du)_{\overline{\alpha}}^{\beta})_j^i   \\
                           &=\sum_{j \in \overline{\alpha}}  \sigma(i,\beta-i)\sigma(j,\overline{\alpha}-j) M_{\overline{\alpha}-j}^{\beta-i}(Du) D_j u^i.
\end{align*}
$u$ is a Lipschitz function which implies
$$\sum_{j \in \overline{\alpha}} D_j((\adj D u)_{\overline{\alpha}}^{\beta})_j^i=0$$
in the sense of distribution. So for all  $\varphi \in C_c^{\infty}(\Omega)$,
$$\int_U M_{\overline{\alpha}}^{\beta}(Du) \varphi d x'=- \int_{U} u^i \sum_{j \in \overline{\alpha}} D_j\varphi((\adj D u)_{\overline{\alpha}}^{\beta})_j^i dx'.$$

Since  $Du_k$ are uniformly Lipschitz functions and  $u_k \rightrightarrows u $ in $\Omega$,
in order to prove the weak convergence of minors in  $L^1$  it suffices to show that for all  $\varphi \in C_c^{\infty}(\Omega)$.
\begin{equation}\label{5.2}
\int _U \varphi M_{\overline{\alpha}}^{\beta}(Du_k)dx' \rightarrow \int_U \varphi M_{\overline{\alpha}}^{\beta}(Du) dx'
\end{equation}

We shall now prove (\ref{5.2}) by induction on the order of the minors. Obviously it holds for  $l=1$ since   $u_k \rightrightarrows u $ in $U$. Suppose that it holds for  $l-1$. Clearly,
\begin{align*}
\int _U \varphi M_{\overline{\alpha}}^{\beta}(Du_k)dx' &= - \int_{U} u_k^i \sum_{j \in \overline{\alpha}} D_j\varphi((\adj D u_k)_{\overline{\alpha}}^{\beta})_j^i dx'  \\
                           &=- \int_{U} u^i \sum_{j \in \overline{\alpha}} D_j\varphi((\adj D u_k)_{\overline{\alpha}}^{\beta})_j^i dx'+ \int_{U} ( u_i-u_k^i) \sum_{j \in \overline{\alpha}} D_j\varphi((\adj D u_k)_{\overline{\alpha}}^{\beta})_j^i dx'.
\end{align*}
By the inductive assumption the first integral on the right tends to
$$- \int_{U} u^i \sum_{j \in \overline{\alpha}} D_j\varphi((\adj D u)_{\overline{\alpha}}^{\beta})_j^i dx',$$
which is equal to
$$\int_U \varphi M_{\overline{\alpha}}^{\beta}(Du) dx'.$$
While the second integral on the right tends to $0$, this proves (\ref{5.2}) for $l$ and therefore the theorem.
\end{proof}

\begin{proof}[\bf Proof of Theorem 1.2]
It suffices to prove the theorem in the case that $w_k$, $w$ are Lipschitz in $\Omega$.
 According to \cite[Lemma 2.3.]{AC} and \cite[Theorem B.3.1.4]{HJL}, it follows that $w_k, w$ can be extended
to be semi-convex  functions as $w^{\ast}_k, w^{\ast}$,
defined in  $\mathbb{R}^n$, such that $w^{\ast}_k\rightrightarrows w^{\ast}$ in $\mathbb{R}^n$.
Let $s>c$, and note that $(\cos \theta I+\sin \theta \partial w^{\ast}_k)(\mathbb{R}^n)=\mathbb{R}^n$ where $\cos \theta = \frac{2s}{\sqrt{1+4s^2}}$. then by the proof of Theorem \ref{thm32},
there exist a rotation transformation $\Theta$ and  Lipschitz functions $u_k: \mathbb{R}^n\rightarrow \mathbb{R}^n$ such that
$$\Theta_{\sharp} G_{u_k}=G_{\partial w^{\ast}_k},
~\mbox{Lip}(u_k)~\leq \max\{2s, \sqrt{\frac{4sc+1}{4s^2-4sc}}\},$$
where $k=0,1,2,...$ and $ w^{\ast}_0:=w^{\ast}$.

First, we need to  show that $u_k \rightrightarrows u$ in any compact $k\subset \mathbb{R}^n$. This result will follow from Arzela-Ascoli Theorem, if we can show that $u_k(x') \rightarrow u(x')$ for any $x' \in \mathbb{R}^n$.
Here we argue by contradiction, assume that there exists  $x' $ such that   $u_k(x')\nrightarrow u(x')$. Let
$$ x_k:=\cos \theta x'-\sin \theta u_k(x'),~~~~y_k:=\sin \theta x'+\cos \theta u_k(x').$$
Then $x'=\cos \theta x_k+\sin \theta y_k$.
and thus
$y_k-y_0=-2s(x_k-x_0)$. Note that $y_k+ cx_k  \in \partial v_k^{\ast}(x_k)$ where $v_k^{\ast}(x):=w_k^{\ast}(x)+c|x|^2$ are convex and Lipschitz in $\mathbb{R}^n$. Hence
 $$(2s-c)x_k=-(y_k+ cx_k)-2sx_0+y_0,$$
 which implies that both $x_k$ and $y_k$ are bounded.

 If  $x_k \nrightarrow x_0$, there exist  $\epsilon_0>0$ and a subsequence  $x_{\lambda_k}$ such that $|x_{\lambda_k}-x_0|\geq \epsilon_0$.
Since $x_{\lambda_k}$ and  $y_{\lambda_k}$  are bounded,
then there exists a subsequence   $x_{\mu_k}$ of  $x_{\mu_k}$  such that
$$x_{\mu_k} \rightarrow x_1,~y_{\mu_k} \rightarrow y_1.$$
By using Proposition \ref{pro2} (\romannumeral2),
 $$w_{\mu_k}^{\ast}(z) \geq w_{\mu_k}^{\ast}(x_{\mu_k})+\langle y_{\mu_k},z-x_{\mu_k}\rangle  - \frac{c}{2} |z-x_{\mu_k}|^2$$
for any  $z \in \mathbb{R}^n$.  Since $w^{\ast}_k\rightrightarrows w^{\ast}$ in $\mathbb{R}^n$, it follows that
$$w^{\ast}(z) \geq w^{\ast}(x_1)+\langle y_1, z-x_1\rangle  - \frac{c}{2} |z-x_1|^2,$$
which implies  $y_1 \in \partial w^{\ast}(x_1)$. So  $x'=\cos \theta x_1+\sin \theta y_1$, and then
\begin{align*}
0=|x'-x'|^2 &=(\cos \theta(x_1-x_0)+\sin \theta (y_1-y_0))^2  \\
                           &\geq \frac{1}{1+4s^2}(4s^2-4sc)(x_1-x_0)^2\\
                           &\geq 0.
\end{align*}
Hence  $x_1=x_0$ which contradicts the assumption that  $|x_{\mu_k}-x_0|\geq \epsilon_0$.

In order to prove the Theorem, according to the fact that $G_{\partial w^{\ast}_k}\llcorner \Omega \times \mathbb{R}^n=G_{\partial w_k}$,
 it is enough to show that  $G_{\partial w^{\ast}_k}\rightharpoonup G_{\partial w^{\ast}} ~\mbox{in}~ \mathcal{D}^n(\mathbb{R}^n \times \mathbb{R})$.
Since  $\Theta_{\sharp}G_{u_k}=G_{\partial w^{\ast}_k}$,
we need only to show that for any ordered multi-indices $\alpha, \beta$ with $|\alpha|+|\beta|=n$ and   $\varphi(x',y') \in C_c^{\infty}(\mathbb{R}^n \times \mathbb{R}^n)$
$$\int \varphi (x',u_k(x')) M_{\overline{\alpha}}^{\beta}(Du_k(x'))dx' \rightarrow \int \varphi(x',u(x')) M_{\overline{\alpha}}^{\beta}(Du(x')) dx',$$
which can be deduced by  Lemma \ref{lma41} and the fact that $u_k \rightrightarrows u$ in any compact $K\subset \mathbb{R}^n$.
\end{proof}
The following comes easily from the standard mollification of $w\in W(\Omega)$.
\begin{corollary}
If $w \in W(\Omega,c)$, then there exists a sequence $\{\omega_k\}^{\infty}_{k=1}\subset W(\Omega',c)\cap C^{\infty}(\Omega', \mathbb{R}^n)$ such that $G_{\partial w_k} \rightharpoonup G_{\partial w}$ in $\mathcal{D}^n(\Omega'\times \mathbb{R}^n)$ for any open set $\Omega'\subset\subset \Omega$. Moreover, it holds for $\Omega$ if $\omega$ is Lipschitz.
\end{corollary}

\begin{corollary}
If $w\in W(\Omega)$ and $n\geq 2$, then $G_{\partial \omega}$ is Lagrangian.
\end{corollary}
\begin{proof}
It suffices to prove the theorem in the case that  $\omega$ is Lipschitz in $\Omega$. Then there exists a sequence $\{\omega_k\}^{\infty}_{k=1}\subset W(\Omega)\cap C^{\infty}(\Omega, \mathbb{R}^n)$ such that $G_{\partial w_k} \rightharpoonup G_{\partial w}$. For any  $\eta \in \mathcal{D}_{n-2}(\Omega\times \mathbb{R}^n)$, since $D^2 \omega_k=(D^2 \omega_k)^{T}$,
$$G_{\partial \omega_k}(\phi \wedge \eta)=0,$$
where $\phi:=\sum_{i=1}^{\infty}dx^i \wedge dy^i$. Which proves $G_{\partial \omega}(\phi \wedge \eta)=0$ and therefore the theorem.
\end{proof}

\section{The proof of Theorem \ref{thm4}.}
 \begin{proof}
 Without loss of generality, we may assume that $w$ is Lipschitz in $\Omega$.
Given  $s>c$, there exist a rotation transformation $\Theta$ and a Lipschitz function $u: D\rightarrow \mathbb{R}^n$ such that $\Theta_{\sharp} G_{u}=G_{\partial w}$,
where $\theta:= \arccos \frac{2s}{\sqrt{1+4s^2}}$. We denote   $g_1(x,y)=\cos \theta x+\sin \theta y$,  $f_2(x',y')=\sin \theta x'+\cos \theta y'$, and $D_B= \{\cos \theta x + \sin \theta y \mid x \in B, y \in \partial w(x)\}$, where $B$ is a Borel subset in $\Omega$.

 On the one hand,
 $$g_{1 \sharp} G_{\partial w} =g_{1 \sharp} \Theta_{ \sharp} G_{ u} =\pi_{\sharp}G_{ u}=\textbf{[}D \textbf{]}.$$
 On the other hand,  for any  $\varphi \in C_c^{\infty}(D)$
  \begin{align*}
g_{1 \sharp} G_{\partial w} (\varphi(x')dx')&= G_{\partial w}(\varphi \circ g_1(x,y)dg_1)\\
                           &=G_{\partial w} (\varphi \circ g_1(x,y) \sum \sigma(\alpha,\overline{\alpha}) \cos^{|\alpha|} \theta \sin^{|\overline{\alpha}|} \theta dx^{\alpha} \wedge dy^{\overline{\alpha}} )\\
                           &=\sum_{i=0}^{n} \cos^{i}\theta \sin^{n-i} \theta \sum _{|\alpha|=i} \sigma (\alpha,\overline{\alpha}) G_{\partial w}^{\alpha \overline{\alpha}} (\varphi \circ g_1 (x,y)).
\end{align*}
Hence
 $$ \textbf{[}D \textbf{]}(\varphi(x') dx' )= \sum_{i=0}^{n} \cos^{i} \theta \sin^{n-i} \theta \sum _{|\alpha|=i} \sigma (\alpha, \overline{\alpha}) G_{\partial w}^{\alpha \overline{\alpha}} (\varphi \circ g_1 (x,y)).$$
Next, we have to show that  for any ordered multi-indices  $\alpha$,
$$\lim_{\epsilon \rightarrow 0} G_{\partial w}^{\alpha \overline{\alpha} }(\varphi_{\epsilon} \circ g_1(x,y))=\int _{\Gamma_{\partial v,B}} \xi^{\alpha \overline{\alpha}}(x,y) d\mathcal{H}^n (x,y),$$
 where  $\varphi_{\epsilon}= \phi_{\epsilon} \ast \chi_{D_B}$. Clearly, $\varphi_{\epsilon} \in C_c^{\infty}(D)$ and  $\varphi_{\epsilon} \rightarrow \varphi:=\chi_{D_B}$  a.e., as  $\epsilon \rightarrow 0$.

 In order to prove the claim it suffices to show that
 $$\nu_{\epsilon}(x,y) \rightarrow \nu(x,y)~~~\mathcal{H}^n~\mbox{a.e.}~(x,y)\in \mathbb{R}^{2n}, $$
 where $\nu_{\epsilon}(x,y)=\chi_{\Gamma_{\partial w,\Omega}} \varphi_{\epsilon}(\cos \theta x+ \sin \theta y ) \xi^{\overline{\alpha} \alpha } (x,y)$ and
$\nu(x,y)=\chi_{\Gamma_{\partial w,\Omega}} \varphi(\cos \theta x+ \sin \theta y ) \xi^{\overline{\alpha} \alpha} (x,y)$.

If  $(x,y) \notin \Gamma_{\partial w,\Omega}$, then $\nu_{\epsilon}(x,y)=0=\nu(x,y)$.
If  $(x,y) \in \Gamma_{\partial w,\Omega}$, there exists $S\subset \Gamma_{\partial w,\Omega}$
such that $\mathcal{H}^n(S)>0$  and  $\nu_{\epsilon}(x,y) \nrightarrow \nu(x,y)$ for any  $(x,y)\in S$.
Let  $S'=g_1(S)$, then
$$\mathcal{H}^n(S)=\mathcal{H}^n (\Gamma_{u,S'})=\int_{S'}|M(Du)| dx'>0.$$
Therefore $\mathcal{H}^n(S')>0$,  which contradicts the assumption that $\varphi_{\epsilon} \rightarrow \varphi $ a.e. Hence the desired result is obtained by the dominated convergence theorem. So
\begin{align*}
\mathcal{L}^n(D_B)       &=\lim_{\epsilon \rightarrow 0} \textbf{[}D \textbf{]}(\varphi_{\epsilon}(x') dx' ) \\
                           &= \lim_{\epsilon \rightarrow 0} \sum_{i=0}^{n} \cos^{i}\theta \sin^{n-i} \theta \sum _{|\alpha|=i} \sigma (\alpha,\overline{\alpha}) G_{\partial w}^{\alpha \overline{\alpha} } (\varphi_{\epsilon} \circ g_1 (x,y)) \\
                           &=\sum_{i=0}^{n} \cos^{i} \theta \sin^{n-i} \theta \sum _{|\alpha|=i} \sigma (\alpha,\overline{\alpha}) \int _{\Gamma_{\partial w,B}} \xi^{\alpha \overline{\alpha}}(x,y) d\mathcal{H}^n (x,y).
\end{align*}
By the Steiner formula for semi-convex in \cite{AP} we have
    $$F_k(w,B)= \sum_{|\alpha|=k} \sigma(\alpha,\overline{\alpha}) \int _{\Gamma_{\partial w,B}} \xi^{\alpha \overline{\alpha}}(x,y) d\mathcal{H}^n (x,y).$$
 In particular,
 \begin{align*}
F_0(w,B) &=\lim_{\epsilon \rightarrow 0} G_{\partial w}^{ \overline{0} 0}(\varphi_{\epsilon} \circ g_1(x,y)) \\
                           &=\lim_{\epsilon \rightarrow 0} \textbf{[}D \textbf{]}(\varphi_{\epsilon}(x') d(\cos \theta x'- \sin \theta u(x')) ) \\
                           &=\int_{D_B} \det(\cos \theta I- \sin \theta Du) d x' \\
                           &=\mathcal{L}^n(B).
\end{align*}
 If $w$ is convex,
 \begin{align*}
F_n(w,B) &=\lim_{\epsilon \rightarrow 0} G_{\partial w}^{ 0 \overline{0}}(\varphi_{\epsilon} \circ g_1(x,y)) \\
                           &=\lim_{\epsilon \rightarrow 0} \textbf{[}D \textbf{]}(\varphi_{\epsilon}(x') d(\sin \theta x'+ \cos \theta u(x')) ) \\
                           &=\int_{D_B} \mbox{det}(\sin \theta I+ \cos \theta Du) d x' \\
                           &=\int_{f_2(D_B)} \mathcal{H}^0(D_B \cap f_2^{-1}(y)) dy,
\end{align*}
where the last equality is deduced by area formula and the fact $\det (\sin \theta I+ \cos \theta Du) \geq0$.
Let  $P:=\{y \in f_2(D_B) \mid \mathcal{H}^0(D_B \cap f_2^{-1}(y)) \neq 1\}$, and fix  $y \in P$.
Then there exist $x_1,x_2 \in B$ such that $x_1 \neq x_2$ and $y \in \partial w(x_1)\cap \partial w(x_2)$,
and hence   $\mathcal{H}^n(P)=0$ by the proof of Theorem 5.11 in \cite{AA}. Therefore
$$F_n(w,B)= \mathcal{L}^n (f_2(D))=\mathcal{L}^n (\partial w(B)),$$
which completes the proof.
\end{proof}
\begin{remark}
It should be observed that the measures $C_n^k F_k$ in the notation of Colesanti-Hug correspond to the Hessian measures $F_k$ in the notation of Trudinger-Wang, and in this paper  we denote $F_k$ in the same way as the latter.
\end{remark}

%---------------------------------------------------------------------------------------%

\section*{Acknowledgments}
\addcontentsline{toc}{chapter}{Acknowledgements}
This work is supported by NSF grant of China ( No.11131005, No.11301400) and the Fundamental Research Funds for the Central Universities(Grant No. 2014201020203).

\bibliographystyle{plain}

\end{document}